\documentclass[12pt]{amsart}
\usepackage{pstricks}
\usepackage{amsfonts}
\usepackage{amssymb}
\usepackage{amsthm}
\usepackage{latexsym,amsmath}
\usepackage{graphicx}
\usepackage{stmaryrd}

\newtheorem{theorem}{Theorem}[section]
\theoremstyle{definition}
\newtheorem{Definition}[theorem]{Definition}

\newenvironment{theorem*}[1]{\medskip
                            \noindent
                            {\bf Theorem #1. }\ %
                            \begingroup \sl}
                            {\endgroup\medskip}

\begin{document}

\title[All large-cardinal axioms are justified]{All large-cardinal axioms not known to be inconsistent with ZFC are justified}

\author[$\mathrm{M^{\lowercase{c}}Callum}$ ]{\textbf{Rupert} $\mathbf{M^{\lowercase{c}}Callum}$ }

\begin{abstract}

In other work we have outlined how, building on ideas of Welch and Roberts, one can motivate believing in the existence of supercompact cardinals. After making this observation we strove to formulate a justification for large-cardinal axioms of greater strength, and arrived at a motivation for the existence of Vop\v{e}nka scheme cardinals. This line of argument also led to the conclusion that a theory $B_0(V_0)$ described in a previous paper of Victoria Marshall implies the existence of a Vop\v{e}nka scheme cardinal $\kappa$ such that $V_{\kappa} \prec V$ (and therefore, in particular, a proper class of extendible cardinals as well).

\bigskip

Marshall left as an open question whether her theory $B_0(V_0)$, whose consistency is implied by the existence of an almost huge cardinal, implied the existence of supercompact or extendible cardinals. Here both questions are resolved positively. In the final section we give an account of how one could plausibly motivate every large-cardinal axiom not known to be inconsistent with choice while stopping short of the point of inconsistency with choice.

\end{abstract}

\maketitle

\section{Introduction}

In \cite{Incurvati2016}, Luca Incurvati defines the scheme $\textsf{RP}_{m,n}$ as follows, for all integers $m, n$ such that $0<m\leq n$. The schema $\textsf{RP}_{m,n}$ is defined to be a schema of sentences in the $n$th-order language of set theory, where for each well-formed formula $\phi$ in the $n$th-order language of set theory with quantified variables of order at most $m$ and free variables $A_{1}, \ldots A_{k}$, the universal closure of $\phi(A_{1}, \ldots A_{k})\implies \exists\alpha \phi^{V_{\alpha}}(A_{1}^{\alpha}, \ldots A_{k}^{\alpha})$ is defined to be one of the sentences in the schema, where $A^{\alpha}$ is defined to be $A \cap V_{\alpha}$ for second-order variables $A$, and $A^{\alpha}$ is defined to be $\{B^{\alpha}\mid B \in A\}$ where $A$ is a variable of order greater than the second order. This completes the definition of the schema $\textsf{RP}_{m,n}$.

\bigskip

For example, $\textsf{RP}_{2,2}+\mathrm{Extensionality}+\mathrm{Foundation}+\mathrm{Separation}$ implies

\bigskip

$\textsf{ZF}+\{$proper class of $\Pi^{1}_{n}$-indescribables $\mid n \in \omega\}$.

\bigskip

But, as was first observed by Reinhardt in \cite{Reinhardt74}, and the first explicit proof of which was given by Tait in \cite{Tait2005a}, $\textsf{RP}_{1,3}$ is inconsistent. Tait tried to resolve this by seeking to motivate restrictions on the formula $\phi$, but Koellner showed in \cite{Koellner2009} that even with these restrictions $\textsf{RP}_{3,4}$ is inconsistent. In \cite{Koellner2009} Peter Koellner extensively examined the question of which reflection principles might be intrinsically justified and formulated a family of reflection principles which were special cases of the ones proposed by Tait and which Koellner showed to be provably consistent relative to an $\omega$-Erd\H{o}s cardinal. Koellner also made the conjecture that all reflection principles which could be formulated and plausibly argued to be intrinsically justified, would either prove to be inconsistent or else provably consistent relative to $\kappa(\omega)$, the first $\omega$-Erd\H{o}s cardinal. I then attempted to take this line of investigation further.

\bigskip

Formulating a notion of an $\alpha$-reflective cardinal for ordinals $\alpha>0$ and seeking to motivate this along lines inspired by remarks in the work of Tait, I showed in \cite{McCallum2013} that if $\kappa$ is $\omega$-reflective then $V_{\kappa}$ satisfies $\textsf{RP}_{m,n}$ for all $m, n$ with some restrictions on $\phi$ slightly more restrictive than Tait's ones. And I showed that it is consistent relative to an $\omega$-Erd\H{o}s cardinal that there is a proper class of $\alpha$-reflective cardinals for each $\alpha>0$. In \cite{McCallum2017} I used similar ideas to motivate the idea of an extremely reflective cardinal, also provably consistent relative to an $\omega$-Erd\H{o}s cardinal, and in fact equivalent to the property of being a remarkable cardinal.

\bigskip

Then the work of Sam Roberts \cite{Roberts2017} appeared seeking to answer Peter Koellner's challenge to formulate an intrinsically justified reflection principle of greater consistency strength than $\kappa(\omega)$. A similar attempt had already been made by Philip Welch in \cite{Welch2014}, where a reflection principle in the second-order language of set theory was described implying the existence of a proper class of Shelah cardinals (and therefore in particular a proper class of measurable Woodin cardinals) and consistent relative to a superstrong cardinal. Let us describe the reflection principle discussed by Sam Roberts in \cite{Roberts2017}.

\bigskip

To explain the reflection principle which Roberts formulates in \cite{Roberts2017}, let us begin by explaining the reflection principle that he calls $\textsf{R}_{2}$. This is an axiom schema in the second-order language of set theory. For each formula $\phi(x_{1}, x_{2}, \ldots x_{m}, X_{1}, X_{2}, \ldots X_{n})$ in the second-order language of set theory, there is an axiom asserting that if $\phi$ holds, then there exists an ordinal $\alpha$ such that $x_{1}, x_{2}, \ldots x_{m} \in V_{\alpha}$, and a ``set-sized" family of classes which contains the classes $X_{1}, X_{2}, \ldots X_{n}$, which is itself coded for by a single class, and which is standard for $V_{\alpha}$ in the sense that every subset $X\subseteq V_{\alpha}$ is such that some class in the family has intersection with $V_{\alpha}$ equal to $X$, such that the formula $\phi$ still holds when the first-order variables are relativised to $V_{\alpha}$ and the second-order variables are relativised to the set-sized family of classes. This completes the description of the axiom schema $\textsf{R}_{2}$. Then Roberts extends the axiom schema as follows. He extends the underlying language so as to include a satisfaction predicate for the second-order language of set theory, and then he extends the axiom schema so as to also include an axiom of the kind described for every formula in this extended language, calling this new axiom schema $\textsf{R}_{S}$. Then he denotes by $\textsf{ZFC2}_{S}$ the result of extending $\textsf{ZFC2}$ -- being the same as $\textsf{ZFC}$ except for having Separation and Replacement as single second-order axioms and also having an axiom schema of class comprehension for every formula in the second-order language of set theory -- by adding the usual Tarskian axioms for the satisfaction predicate and extending the class comprehension axiom schema to include axioms involving formulas in the extended language. Then he proceeds to investigate the theory $\textsf{ZFC2}_{S}+\textsf{R}_{S}$. This completes the description of the reflection principle which Roberts considers. He shows that the theory $\textsf{ZFC2}_{S}+\textsf{R}_{S}$ proves the existence of a proper class of 1-extendible cardinals and is consistent relative to a 2-extendible cardinal.

\bigskip

I have explored elsewhere the question of whether this reflection principle is intrinsically justified, and have also described how this line of thought could plausibly be taken further to motivate a reflection principle equivalent to the existence of a supercompact cardinal. Let me briefly explain how that can be done. Given a level $V_{\kappa}$, we can consider structures of the form $(V_{\kappa}, V_{\lambda})$ with $\lambda>\kappa$ and consider some formula $\phi$ in a two-sorted language holding in such a structure relative to a certain finite collection of parameters. It is natural to posit that there should exist a ``set-sized" reflecting structure, containing all the parameters, whose first component is $V_{\alpha}$ for some $\alpha<\kappa$ and whose second component is ``set-sized" in the sense of having cardinality less than $\beth_{\kappa}$ and furthermore such that the transitive collapse of the second component is of the form $V_{\beta}$ for some $\beta>\alpha$. (Here the collapsing map may not be injective.) A level $V_{\kappa}$ satisfies this form of reflection if and only if $\kappa$ is supercompact, as can be seen from Magidor's characterisation of supercompactness. Let us assume for the sake of argument, for the rest of this paper, that these kinds of considerations can be taken as a good motivation for the view that supercompact cardinals are intrinsically justified. Can one use further ideas to motivate large cardinals of still greater strength being intrinsically justified?

\bigskip

In Section 2 I shall tell the story of how I followed a line of thought seeking to find an intrinsic justification for extendible cardinals, building on these ideas, and arrived at a motivation for the existence of Vop\v{e}nka scheme cardinals. This line of argument can be used to obtain a proof that the theory $B_0(V_0)$, defined in \cite{Marshall89}, implies the existence of a Vop\v{e}nka scheme cardinal $\kappa$, such that $V_{\kappa} \prec V$. In particular, the theory $B_0(V_0)$ implies the existence of a proper class of extendible cardinals. Marshall raised in \cite{Marshall89} the question of whether $B_0(V_0)$ implies the existence of supercompact and extendible cardinals, and here both questions are resolved positively. In Section 3 we try to use ideas based on the Marshall's paper to motivate all large cardinals not known to be inconsistent with choice, but not the ones known to be inconsistent with choice.

\section{Vop\v{e}nka scheme cardinals}

Suppose that $\kappa$ is a supercompact cardinal. We define a normal proper filter $F$ on $\kappa$ consisting of those $X \subseteq \kappa$ such that, for some $\delta\geq\kappa$, $\kappa \in j(X)$ for every embedding $j:V \prec M$ witnessing the $\gamma$-supercompactness of $\kappa$ for a $\gamma\geq\delta$. This filter contains the set $X_{\alpha}$, consisting of all $\alpha$-extendible cardinals less than $\kappa$, for each $\alpha<\kappa$. We want to find a sufficient condition for being able to conclude that $\bigcap_{\alpha<\kappa} X_{\alpha} \in F$, so that $V_{\kappa}$ will be a model for the existence of a proper class of extendible cardinals. Let us describe the theory $B_0(V_0)$ of Marshall's paper \cite{Marshall89}.

\bigskip

It is a theory in the first-order language of set theory with the additional constant symbol $V_0$. First, any axiom of $\textsf{ZFC}$, or its relativisation to $V_0$ is taken as an axiom. Also Extensionality and Foundation are taken as axioms. And if $\phi$ is a formula with at least one free variable $x$, which does not contain $u$ or $\kappa$ free, then $\phi(A) \implies \exists \kappa \in \mathrm{On} \exists u (u \cap V_0 = R_\kappa \wedge \forall x \forall y (x, y \in u \implies [x,y] \in u) \wedge \phi^{V_0}(A^u)$ is taken as an axiom, where we define On to be the set of ordinals in $V_0$, $A^u=A\cap u$ if $A\in\mathcal{P}(V_0)$, and $A^u=\{x^u:x\in A\cap u\}$ if $A \notin\mathcal{P}(V_0)$, and $[x,y]=x\times\{0\}\cup y\times\{1\}$. Clearly if $V_{\rho}$ is a model of $B_0(V_0)$ with the constant symbol $V_0$ interpreted by $V_{\kappa}$, then $\kappa$ is a supercompact cardinal in $V_{\rho}$ with $V_{\kappa}\prec V_{\rho}$.

\begin{theorem} Suppose that $\kappa$ is a supercompact cardinal in $V_{\rho}$ with $V_{\kappa} \prec V_{\rho}$ and such that $(V_{\kappa}, V_{\rho})$ is a model of $B_0(V_0)$. Then $V_{\kappa}$ is a model for the assertion that there is a proper class of extendible cardinals. \end{theorem}

\begin{proof}

Suppose that $X \in L_{1}(V_{\kappa})$ or $X \in L_{2}(V_{\kappa})$, where we require $X$ to be parameter-free definable in the latter case. In either case, define $j(X)$ to be the element of $L_{1}(V_{\rho})$ or $L_{2}(V_{\rho})$ defined by the same formula as the formula defining $X$, the choice of formula doesn't matter.

\bigskip

For all $X \in L_{1}(V_{\kappa})$, we have $X \in F$ iff $\kappa \in j(X)$, so that $F \cap L_{1}(V_{\kappa})$ is ordinal-definable in $V_{\rho}$. And there is a normal proper filter $F''$ on $\kappa$ which is a parameter-free definable of $L_2(V_{\kappa})$, such that $\langle X_\alpha \mid \alpha<\kappa \rangle \subseteq F'' \subseteq F$. Define a filter $F'$ on $\rho$ by $F':=j(F'')$. $F'$ is also a normal proper filter, and $j(X_{\alpha}) \in F'$ for each $\alpha<\kappa$, recalling that when $\alpha<\kappa$, the notation $X_\alpha$ denotes the set of $\alpha$-extendible cardinals less than $\kappa$. The assertion that for every $X$ which is in $F''$ and ordinal-definable in $V_{\kappa}$, $\kappa \in j(X)$, is true relative to $V_{\rho}$, and we are assuming that $(V_{\kappa}, V_{\rho})$ is a model of $B_0(V_0)$. For a sufficiently large $n>0$ we can find a $Y\in F$ such that for all $\delta_1<\delta_2<\kappa$ with $\delta_1, delta_2 \in Y$, we have $V_{\delta_1} \prec_{n} V_{\delta_2} \prec_{n} V_{\kappa}$. We can assume $Y \subseteq \Delta_{\alpha<\kappa} X_{\alpha}$ and that $Y \in L_1(V_\kappa)$. So if $\beta \in Y$ then $\beta \in X_\alpha$ for all $\alpha<\beta$. We want to claim that $Y$ can be chosen to satisfy all the previously stated hypotheses together with the statement that even if $\beta'<\beta$, $\beta' \in Y$, we still have $\beta' \in X_\alpha$ for all $\alpha<\beta$. This is because we can choose $Y$ in such a way that whenever $\delta_1<\delta_2<\kappa$ and $\delta_1, \delta_2 \in Y$ there is a $\Sigma_{n}$-elementary embedding $j:V_{\delta_2} \prec_{n} V_{\kappa}$ with $j(\delta_1)=\delta_2$. Choosing $n$ to be sufficiently large will yield the desired statement. From this we get $Y\subseteq\bigcap_{\alpha<\kappa} X_\alpha$ and so $\bigcap_{\alpha<\kappa} X_\alpha$ is non-empty as claimed. Thus we get the desired conclusion that $\kappa$ is a limit of cardinals that are extendible in $V_{\kappa}$, in fact we can even say $C^{(n)}$-extendible in $V_{\kappa}$, rather than just extendible in $V_{\kappa}$, for every $n$ and so by results of \cite{Bagaria2012} we conclude that $\kappa$ is a Vop\v{e}nka scheme cardinal. We have shown if $(V_{\kappa}, V_{\rho})$ is a model for $B_0(V_0)$ then $\kappa$ is a Vop\v{e}nka scheme cardinal with $V_{\kappa} \prec V_{\rho}$.
\end{proof}

This completes our description of our initial line of thought leading to the conclusion that Vop\v{e}nka scheme cardinals are justified.

\section{Justification for all large cardinals not known to be inconsistent with ZFC}

For the purpose of the argument discussed in this section, we will need to present the definition of Marshall's theory $B_0(V_0^0,V_0^1, \ldots V_0^{n-1})$ discussed in \cite{Marshall89}. It is a theory in the first-order language of set theory with constant symbols $V_0^0, V_0^1, \ldots V_0^{n-1}$. The axioms are the same as for $B_0(V_0)$ except that now the relativization of $\phi$ to $V_0^k$ for each axiom $\phi$ of ZFC is taken as an axiom for all $k$ such that $0\leq k<n$. And the reflection principle is now $\phi(A) \implies \exists \kappa \in \mathrm{On} \exists u (V_0^0 \cap u = R_\kappa \wedge (V_0^1)^u=V_0^0 \wedge (V_0^2)^u=V_0^1 \wedge \ldots \wedge (V_0^{n-1})^u=V_0^{n-2} \wedge \forall x \forall y (x, y \in u \equiv [x,y] \in u) \wedge \phi^{V_0^{n-1}}(A^u)$, where On is the set of ordinals in $V_0^0$. We can assume that the axiom of extensionality occurs as a conjunct of $\phi(A)$ and therefore speak of the embedding witnessing each instance of reflection, and we shall do so in what follows.

\bigskip

We introduce the following definitions.

\begin{Definition} A cardinal $\kappa$ is said to be an $n$-Marshall cardinal, for an ordinal $n\in\omega$ such that $n>0$, if there exist $\kappa_0<\kappa_1<\ldots<\kappa_{n-1}<\kappa$ such that to this finite sequence of ordinals corresponds a natural model of $B_0(V_0^0,V_0^1,\ldots V_0^{n-1})$. A cardinal $\kappa$ is said to be a 0-enormous cardinal if it is an $n$-Marshall cardinal for every $n\in\omega\setminus\{0\}$. (Note that a 0-enormous cardinal is bounded above in consistency strength by a totally huge cardinal.) A cardinal $\kappa$ is said to be an $\alpha$-enormous cardinal, for an ordinal $\alpha>0$, if there exists a sequence $\langle \kappa_\beta : \beta<\alpha \rangle$ with $\kappa_0=\kappa$ such that for all $n>0$, either $\alpha>n$ and every sequence of cardinals $\kappa_{\beta_0}<\kappa_{\beta_1}<\ldots<\kappa_{\beta_{n-1}}<\kappa_{\beta_n}$ corresponds to a natural model of $B_0(V_0^0,V_0^1,\ldots V_0^{n-1})$, or $\alpha\leq n$ and for every sequence of cardinals $\kappa_{\beta_0}<\kappa_{\beta_1}<\ldots<\kappa_{\beta_m}$ with $m\leq n$, there exist cardinal $\rho_0<\rho_1<\ldots<\rho_{n-m-1}<\kappa_0$, such that the sequence $\rho_0<rho_1<\ldots<\rho_{n-m-1}<\kappa_{\beta_0}<\kappa_{beta_1}<\ldots<\kappa_{\beta_m}$ corresponds to such a model.  \end{Definition}

Hugh Woodin has elsewhere defined an enormous cardinal to be a cardinal $\kappa$ such that there exist ordinals $\lambda, \gamma$ such that $\kappa<\lambda<\gamma$ and $V_{\kappa} \prec V_{\lambda} \prec V_{\gamma}$ and there is an elementary embedding $j:V_{\lambda+1} \prec V_{\lambda+1}$ with critical point $\kappa$. We shall eventually show that a cardinal is enormous if and only if it is $\omega+1$-enormous. Our goal shall be to show that $\textsf{ZFC}$+``$\kappa$ is an $\omega+1$-enormous cardinal" implies the $V_{\kappa}$ is a model for the existence of a proper class of cardinals with the large-cardinal property $\phi$, for every large-cardinal property $\phi$ that has been previously considered such that the existence of a cardinal with this property is not known to be inconsistent with $\textsf{ZFC}$, with the exception of the property of being an enormous cardinal.

\bigskip

So suppose that the sequence $\langle \kappa_\alpha : \alpha \leq \omega \rangle$ witnesses that $\kappa:=\kappa_0$ is an $\omega+1$=enormous cardinal. For each finite ordinal $n>0$, the sequence $\langle V_{\kappa_{i+1}} : i<n \rangle$ can serve as the sequence $\langle V_0^{i}: i<n \rangle$, in a model for the theory $B(V_0^0,V_0^1, \ldots V_0^{n-1})$, and the critical point of the embedding witnessing each reflection axiom can always be chosen to be $\kappa_0$, and further, for each theory in the sequence one may choose an embedding which witnesses reflection for all formulas simultaneously, and these embeddings may be chosen so as to cohere with one another.

\bigskip

Then all of these embeddings may be glued together to yield an embedding $j:V_{\lambda} \prec V_{\lambda}$ where $\lambda$ is the supremum of the $\kappa_{i}$'s. This embedding (together with its image under iterates of its own extension to $V_{\lambda+1}$) witnesses that each $\kappa_{i}$ is an $I_3$ cardinal.

\begin{theorem} Assume $\textsf{ZFC}$+``$\kappa_0$ is an $\omega+1$=enormous cardinal" with the notation $\langle \kappa_{i}:i \leq\omega\rangle$ as before. Let $\lambda:=\mathrm{sup}\{\kappa_{i}:i \in\omega\}$, $\gamma:=\kappa_\omega$, and $\kappa:=\kappa_{0}$. Then $V_{\kappa}\prec V_{\lambda}\prec V_{\gamma}$, and $V_{\kappa}$ is a model for the existence of a proper class of $I_0$ cardinals and also a proper class of each of the large cardinals considered by Hugh Woodin in \cite{Woodin2011}. \end{theorem}

\begin{proof}

In fact, using the hypothesis $V_{\kappa_{i}} \prec V_{\gamma}$, we get reflection for any formula with parameters from anywhere in $V_{\gamma}$, even with rank greater than or equal to $\lambda$, and we have embeddings witnessing the reflection of the kind described in each theory $B_0(V_0^0, V_0^1, \ldots V_0^{n-1})$ for each $n$, with the constant symbols $V_0^k$ interpreted by $V_{\kappa_{k+1}}$ for all integers $k$ such that $0\leq k<n$. Consider first the case of parameters from $V_{\lambda+1}$, in this case for each $n$ the same choice of embedding will work for all formulas, and the family of restrictions of these embeddings to $V_{\kappa_{n-1}}$, where $n$ is the positive finite ordinal corresponding to the embedding, can all be glued together to obtain an embedding with domain $V_{\lambda}$, and this determines an embedding with domain $V_{\lambda+1}$. This embedding withnesses that $\kappa$ is an $I_1$ cardinal, and in particular this embedding induces a unique $\omega$-huge embedding $j:V_{\gamma} \prec M_{\gamma}$ with critical point $\kappa$. For a model $K$ of $\textsf{ZF}$ such that $V_{\lambda+1} \cup \gamma \subseteq K \subseteq V_{\gamma}$, definable in $V_{\gamma}$ from ordinals fixed by $j$, in which every element is ordinal definable in $V_{\gamma}$ from elements of $V_{\lambda+1}$, and the same definition with ordinal parameters works relative to $K$ as well as $V_{\gamma}$, we can make use of the same argument using parameters from $K$ to obtain an embedding $j:K \prec K$. So we see that we obtain an embedding witnessing that $\kappa$ is an $I_0$ cardinal, and an embedding of the kind described in Laver's axiom and all the other axioms stronger than $I_0$ considered by Hugh Woodin in \cite{Woodin2011}. Moreover in each case we obtain that in $V_{\kappa}$ there is a proper class of cardinals $\delta$ which are the critical point of such an embedding. This completes our argument that from our stated assumption we obtain a model for every large-cardinal axiom not known to be inconsistent with choice.

\end{proof}

In this way, using ideas built on those in Marshall's paper, one can provide motivations for all large cardinals not known to be inconsistent with choice, while still having principled reasons to stop short of the point of inconsistency with choice.

\pagebreak[4]

\end{document}